\documentclass{amsart}
\usepackage{enumerate}
\usepackage{color}
\usepackage[mathcal]{euscript}
\usepackage{mathrsfs}
\usepackage[T1]{fontenc}
\usepackage{graphicx}  
\usepackage[colorlinks=true]{hyperref}

\newtheorem{thm}{Theorem}[section]
\newtheorem{cor}[thm]{Corollary}
\newtheorem{lem}[thm]{Lemma}

\theoremstyle{definition}
\newtheorem{defn}[thm]{Definition}
\newtheorem{rem}[thm]{Remark}

\numberwithin{equation}{section}

\DeclareMathOperator{\dist}{dist}
\DeclareMathOperator{\diam}{diam}

\newcommand{\A}{\mathcal{A}}

\newcommand{\N}{\mathbb{N}}

\newcommand{\R}{\mathbb{R}}

\newcommand{\set}[1]{\left\{#1\right\}}
\newcommand{\eps}{\varepsilon}

\begin{document}
\title{Dendrites and chaos}

\author{TOMASZ DRWI\k{E}GA}

\address{Faculty of Applied Mathematics, AGH University of Science and Technology, Address\\
al. A. Mickiewicza 30, 30-059 Krak\'ow, Poland\\
drwiega@agh.edu.pl}
\email{drwiega@agh.edu.pl}

\maketitle
\begin{abstract}
We answer the two questions left open in [Z.~Ko\v{c}an, \emph{Chaos on one-dimensional compact metric spaces}, Internat. J. Bifur. Chaos Appl. Sci. Engrg. \textbf{22}, article id: 1250259 (2012)] i.e. whether there is a relation between $\omega$-chaos and distributional chaos and whether there is a relation between an infinite LY-scrambled set and distributional chaos for dendrite maps. We construct a continuous self-map of a dendrite without any DC3 pairs but containing an uncountable $\omega$-scrambled set. To answer the second question we construct a dendrite $\mathcal{D}$ and a continuous dendrite map without an infinite LY-scrambled set but with  DC1 pairs.
\end{abstract}

\section{Introduction}

Ko\v{c}an \cite{koc12} studied relations between some notions of chaos for continuous maps on dendrites. He concluded with a list of open problems:
\begin{enumerate}[(1)]
\item Does the existence of an uncountable $\omega$-scrambled set imply distributional chaos?
\item Does the existence of an uncountable $\omega$-scrambled set imply existence of an infinite LY--scrambled set?
\item Does distributional chaos imply the existence of an infinite LY--scrambled set?
\end{enumerate}
Here we focus on the relation between $\omega$-chaos and distributional chaos on dendrites and we answer two of them in the negative. To construct appropriate example to answer one of the questions we use properties of spacing shifts (see \cite{lau72}). In this paper after reviewing the basic definitions and properties of dynamical systems we recall a variety of definitions of chaos. We also state a few useful properties on the Gehman dendrite and we  construct the self-map of dendrite which does not have DC3 pairs, but has uncountable $\omega$-scrambled set. Moreover, we show how construct extension of the shift without DC3 pairs to the mixing shift with no DC3 pairs. In the last section of this paper we give an example of shift with  DC1 pair but without an infinite LY-scrambled set.

\section{Definitions and notations} \label{sec:def}

Throughout this paper $\N$ denotes the set $\{1,2,3, \dots \} $ and $\N_0 = \N \cup \{0\}.$ By a \emph{dynamical system} we mean a pair $(X,f)$, where $X$ is a compact metric space with a fixed metric $\rho$ and $f$ is a continuous map from $X$ to itself. The \emph{orbit} of $x \in X$ is the set $O(x):=\{f^{k}(x): k\geq 0\}$, where $f^k$ stands for the $k$-fold composition of $f$ with itself. For $x \in X$ the $\omega$-$limit$ $set$ is the set $$ \omega_f(x):= \{y\in X: \exists{(n_k)^{\infty}_{k=1} \subseteq \N, n_k \nearrow \infty} ~  \lim_{k \to \infty}f^{n_k}(x)=y\}.$$ A set $A \subset X$ is \emph{invariant} under $f$ if $f(A) \subset A$. We say that   $A \subseteq X$ is a \emph{minimal} for $f$ if $A$ is nonempty, closed, invariant under $f$, and does not contain any proper subset which satisfies these three conditions. We say that the dynamical system $(X,f)$ is \textit{minimal} if $X$ is a minimal set for $f$. It is known that $(X,f)$ is minimal if and only if every $x\in X$ has dense orbit or, equivalently, $\omega_f (x) =X$ for each $x \in X$. 

We say that a system $(X, f)$ is 
\begin{enumerate}[(1)]
\item \emph{transitive} if for any nonempty open sets $U,V \subseteq X$  there exists $n>0$ such that $f^n(U) \cap V \neq \emptyset$;
\item \emph{weakly mixing} if $(X\times X, f \times f)$ is transitive;
\item \emph{mixing} if for every nonempty open sets $U,V \subseteq X$,  $f^n(U) \cap V \neq \emptyset$ for sufficiently large $n$;
\item \emph{exact} if for every nonempty open set $U \subset X$ there is $n>0$ such that $f^n(U)=X$.
\end{enumerate}

Observe that $x \in X$ is \emph{a minimal point} for $f$ if for any open and nonempty set $U \subseteq X$ there exists a positive number $m \in \N$ such that for any $i\geq 0$ there exists $j \in [i, i+m]$ such that $f^{j}(x) \in U$. By $N(U,V)$ we denote the set of all $i\in \N$ such that $ f^{i}(U) \cap V \neq \emptyset.$

A point $x \in X$ is regularly recurrent if for any neighborhood $U$ of $x$ there exists $k\in \N$ such that for every $i\in \N$ we have $f^{ik}(x)\in U.$ It is known (see \cite{blo92}) that every regular recurrent point is an element of its own $\omega$-$limit$ set which is minimal. By $I$ we denote the interval $[0,1].$ An \emph{arc} is any topological space homeomorphic to $I$. A \emph{continuum} is a nonempty connected compact metric space. A \emph{dendrite} is a locally connected continuum containing no subset homeomorphic to the circle. A point $x$ of a continuum is an \emph{end point} if for every neighborhood $U$ of $x$ there exists a neighborhood $V$ of $x$ such that $V \subseteq U$ and boundary of $V$ is an one-point set.

Now let us present some standard notation related to symbolic dynamics. Let $\A$ be any finite set (an \emph{alphabet}) and let $\A^*$ denote the set of all finite \emph{words} over $\A$ including the empty word. For any word $w \in \A^*$ we denote by $|w|$ the length of $w$, that is the number of letters which form this word. If $w$ is the empty word then we put $|w|=0$. An \emph{infinite word} is a mapping $w: \N \to \A$, in other words it is an infinite sequence $w_1 w_2 w_3 \dots$ where $w_i \in \A$ for any $i \in \N$. The set of all infinite words over an alphabet $\A$ is denoted by ${\A}^{\N}$. We endow  ${\A}^{\N}$ with the product topology of discrete topology on $\A.$ By $0^k=0 \dots 0$ for $k>1$ we denote word consisting of $k$ zeros  and by $0^{\infty}$ we denote the infinite word $0^{\infty}=000 \dots \in {\A}^{\N}.$ If $x \in \A^\N$ and $i,j \in \N$ with $i \leq j$ then we denote $x_{[i,j)}=x_{i}x_{i+1} \dots x_{j-1}$ (we agree with that $x_{[i,i)}$ is empty word) and given $X \subset {\A}^{\N}$ by $\mathcal{L}(X)$ we denote \emph{the language} of $X$, that is, the set $\mathcal{L}(X):=\{x_{[1,k)}: x \in X, k> 0 \}$.  We write $\mathcal{L}_n(X)$ for the set of all words of length $n$ in $\mathcal{L}(X).$ We say that the word $u \in  \A^*$ \emph{appears} in $z\in {\A}^{\N}$ (the same for $z\in \A^*$) if there are $i \leq j$ such that $j-i=|u|$ and $z_{[i,j)}=u.$ We also write $u \sqsubset z$ and say that $u$ is a \emph{subword} of $z$. If $(u_k)$ is a sequence of words such that $|u_k|\longrightarrow \infty$ then we write $z=\lim_{k\to \infty} u_k$ if the limit $z=\lim_{k\to \infty} u_k 0^\infty$ exists in $\A^\N$. Let $n \in \N$ and $\sigma$ a shift map defined on ${\A}^{\N}$ by $$(\sigma(x))_i=x_{i+1} \text{  for } i\in \N.$$ By  $\Sigma^{+}_{n}$ we denote a dynamical system formed by $ (\{0, \dots, n-1\}^{\mathbb{N}} , \sigma).$ If $S \subset {\A}^{\N}$ is nonempty, closed and $\sigma$-invariant then $S$ together with the restriction $\sigma |_{S} \colon S \to S$ (or even the set $S$) is called a \emph{subshift} of $\Sigma^{+}_{n}.$ Recall that the space ${\A}^{\N}$  is metrizable by the metric  $\rho: {\A}^{\N} \times {\A}^{\N} \to \R$ given for $x,y \in {\A}^{\N}$ by $$ \rho(x,y)= \begin{cases} 2^{-k}, &\text{if } x \neq y, \\ 0, &\text{otherwise} \end{cases} $$ where $k$ is the length of maximal common prefix of $x$ and $y$, that is $k=\max \{ i \geq 1: x_{[1,i)}=y_{[1,i)}\}.$ 

Let $X$ be a subshift. It is well known that $(X,\sigma)$ is
\begin{enumerate}[(1)]
\item \emph{weakly mixing} if  for any $m>0$ and any words $u_1, u_2, v_1,v_2$ from $\mathcal{L}(X)$ with length $m$ there are words $w_1,w_2$ such that $|w_1|=|w_2|$ and such that $u_1w_1v_1, u_2w_2v_2 \in \mathcal{L}(X);$
\item \emph{mixing} if for any $u, v  \in \mathcal{L}(X)$ there is $N>0$ such that for any $n \geq N$ there exists a word $w$ of length $n$ such that $uwv \in \mathcal{L}(X);$
\item \emph{exact} if for any $u \in \mathcal{L}(X)$ there is $N>0$ such that for every $v \in \mathcal{L}(X)$ there exists a word $w$ of length $n \geq N$ such that $uwv \in \mathcal{L}(X)$.
\end{enumerate}

Given $w \in  \A^*$ by $C[w]=\{x \in {\A}^{\N}: x_{[1,|w|+1)}=w\}$ we denote the so-called \emph{cylinder set} of $w$) and by $C_{X}[w]=C[w]\cap X$ we denote \emph{trace of cylinder set} $C[w]$ on $X \subseteq {\A}^{\N}$. 
The collection of all cylinder sets form a basis of the topology of ${\A}^{\N}$.

\begin{defn}[Number of occurrences]
Let $X$ be a shift space over $\A$. For every symbol $a \in \A$ and word $x \in {\A}^*$ we define \emph{number of occurrences $\|x\|_a$ of the symbol $a$ in $x.$}
Let $x=x_1\dots x_k \in \mathcal{L}(X)$ and let $\|x\|_a$ denote the number of $a$'s in $x$, that is
$$\|x\|_a=|\{1\leq j \leq k: x_j=a\}|.$$
\end{defn}

For any $P \subseteq \N$ define
$$ \Sigma_P=\{s \in \Sigma^{+}_{2}: (s_i=s_j=1) \Longrightarrow |i-j|\in P \cup \{0\}\}$$
which is a subshift. We will call a subshift defined in this way the \emph{spacing shift} since it restricts the \emph{spacings} between 1's (see \cite{lau72}). Recall a subset $P$ of $\N$ is called \emph{thick (or replete)} if it contains arbitrarily long blocks of consecutive integers. In other words $P$ is \emph{thick} if and only if  
$$(\forall n\in \N)(\exists m \in \N) \{m,m+1,\dots,m+n\} \subseteq P.$$ We recall that $\Sigma_P$ is weakly mixing iff $P$ is thick (see \cite{lau72} or \cite{ban13}).
\begin{defn}[Densities]
Let $A \subset \N$ be a subset of the positive integers.\\
Its \emph{upper density} is defined
$$\bar{d}(A)=\limsup_{n \to \infty} \frac{\left|A \cap \{1, 2, \dots, n\}\right|}{n}.$$
Its \emph{lower density} is defined
$$\underline{d}(A)=\liminf_{n \to \infty} \frac{\left|A \cap \{1, 2, \dots, n\}\right|}{n}.$$
Its \emph{density} is defined 
$$d(A)=\lim_{n \to \infty} \frac{\left|A \cap \{1, 2, \dots, n\}\right|}{n}.$$
\end{defn}

We restate a version of Mycielski's theorem (\cite{myc64}, Theorem 1) that will play a crucial role in the proof of our Theorem \ref{thm:omega3} .
\begin{defn}[Cantor set]
Let $X$ be a metric space. The set $X$ is a \emph{Cantor set} if it is non-empty, compact, totally disconnected and has no isolated points.
\end{defn}

\begin{defn}[Mycielski set]
Let $X$ be a complete metric space. We call $S \subseteq X$ a $Mycielski$ $set$ if it has the form $S=\bigcup_{j=1}^{\infty} C_j$ with $C_j$ a Cantor set for every $j.$ 
\end{defn}

\begin{thm}[Mycielski]
	Let $X$ be a perfect complete metric space and for each $n \in \N$ let $R_n \subset X^n$ be a residual subset of $X^n.$  Then there is a dense Mycielski set $S \subset X$ such that $(x_1,x_2, \dots, x_n) \in R_n$ for any $n \in \N$ and any pairwise distinct points $x_1,x_2, \dots,x_n \in S$.
\end{thm}

\section{Definition of chaos} \label{sec:chaos}
In the following, we provide definitions of some chaotic behavior of maps.
\subsection{Topological chaos}
Let $\eps>0$ and $n \in \N.$ A set $A \subseteq X$ is $(f,n,\eps)$--\emph{separated} if for each $x,y \in A$ with $x \neq y$ there is an integer $0\leq i <n$, such that $\rho(f^i(x), f^i(y))>\eps$. Let $s(f,n,\eps)$ denote the maximal cardinality of an $(f,n,\eps)$--separated set. The \emph{topological entropy} of $f$ is $$h(f)=\lim_{\eps \to 0^+} \limsup_{n \to \infty} \frac{1}{n} \log s(f,n,\eps) \in [0,\infty]=[0,\infty) \cup \{\infty\}.$$ 
We say that $f$ is \emph{topologically chaotic} (abbreviated PTE) if $f$ has positive topological entropy (see \cite{walt82, down11}).

\subsection{Li-Yorke chaos}
A set $S \subseteq X$ is called \emph{LY--scrambled} for $f$, if it contains at least two points and for any  $x,y \in S$ with $x \neq y$, we have
$$\liminf_{n \to \infty} \rho(f^n(x), f^n(y))=0$$ and 
$$\limsup_{n \to \infty} \rho(f^n(x), f^n(y))>0.$$ 
We say that $f \colon X \to X$ is  $\text{LY}$ chaotic if there exists an uncountable LY--scrambled set  (see \cite{liyor75}).

\subsection{\texorpdfstring{$\omega$-chaos}{omega}}

A set $S \subseteq X$ is called $\omega$-scrambled for $f$ if it contains at least two points and for any $x,y \in S$ with $x\neq y,$ we have
 \begin{enumerate}[(i)]
\item $\omega_f (x) \setminus \omega_f (y)$ is uncountable,
\item $\omega_f (x) \cap \omega_f (y)$ is nonempty,
\item $\omega_f (x) $ is not contained in the set of periodic points.
\end{enumerate} 

We say that $f$ is $\omega$-$chaotic$ if there is an uncoutable  $\omega$-scrambled set for $f$ (see \cite{li93}).

\subsection{Distributional chaos}
This type of chaos was introduced in \cite{sch94}.  Given $f$, $x,y \in X$ and a positive integer $n$, define the distribution function $F^{(n)}_{xy} \colon  (0, diamX] \to [0,1]$ by  
$$F^{(n)}_{xy}(t)=\frac{1}{n} \left| \{0\leq i <n \colon \rho \left( f^i(x),f^i(y)\right)<t\} \right|.$$
Then $F^{(n)}_{xy}$ is a left-continuous nondecreasing function. We define the \emph{lower distribution function} $F_{xy}$ and the \emph{upper distribution function}  $F_{xy}^{*}$  generated by $f$, $x$ and $y$ by
$$F_{xy}(t)=\liminf_{n \to \infty} F^{(n)}_{xy}(t)$$ and
$$F_{xy}^{*}(t)=\limsup_{n \to \infty} F^{(n)}_{xy}(t).$$

We extend $F_{xy}$ and $F_{xy}^{*}$ to the whole real line by setting  $F_{xy}(t)= F_{xy}^{*}(t)=0$ for $t \leq 0$ and $F_{xy}(t)=F_{xy}^{*}(t)=1$ for $t$ which is strictly larger than the diameter of $X$. Clearly, $F_{xy}(t)\leq F_{xy}^{*}(t)$ for every $t \in \R.$
We say that a pair $x,y \in X$ is: 
\begin{enumerate}[(DC1):]
\item if $F_{xy}^{*}(t)=1$ for all $t>0$ and there is $s>0$ such that $F_{xy}(s)=0$,
\item if $F_{xy}^{*}(t)=1$ for all $t>0$ and there is $s>0$ such that $F_{xy}(s)<1$,
\item if there are $a<b$ such that $F_{xy}^{*}(t)>F_{xy}(t)$ for every $t \in (a,b)$.
\end{enumerate}

A set $S \subseteq X$ is  \emph{distributionally chaotic of type 1, 2, 3}  for $f$, if it contains at least two points and for any $x,y \in S$ with $x \neq y$ and the pair $(x,y)$ satisfies the condition (DC1), (DC2), (DC3) respectively.
If there is an uncountable a distributionally chaotic set of type 1, 2, 3 for $f$, then we say that $f$ exhibits \emph{distributional chaos of type 1, 2, 3}, briefly, DC1-chaotic (DC2-chaotic, DC3-chaotic respectively). 
Note that the weaker notions than DC1 distributional chaos denoted by DC2 and DC3 were introduced by Sm\'{\i}tal and ~\v{S}tef\'ankov\'a (see \cite{smiste03}).

\section{The Gehman Dendrite}\label{sec:lem}
Let us recall the construction of a continuous dendrite map from \cite{kokor11}. Let $G$ be the Gehman dendrite (see \cite{geh25}). It is well-known that the Gehman dendrite can be written as the closure of the union of countably many arcs in $\R^{2}$ that is $B_0=[p,p_0], B_1=[p,p_1]$, and for every $n \in \N$, $B_{i_1 i_2  \dots i_{n+1}}=[p_{i_1 i_2 \dots i_n}, p_{i_1 i_2 \dots  i_{n+1}}]$ where every $i_k$ is either 0 or 1. Let $E$ denote the set of end points of $G$. With every point $x \in E$ we can uniquely associate a sequence of zeros and ones $i_1 i_2 i_3 \dots$ in such way that the limit of the codes of the arcs converging to the point.
 
 We define a continuous map $g$ on a dendrite $G$ in the following way. Let $g(B_0)=g(B_1)=\{p\}$. For every $i_1,i_2, \dots, i_n$, let $g |_{B_{i_1 i_2  \dots  i_n}}: B_{i_1 i_2  \dots  i_n} \to B_{i_2 i_3  \dots i_n}$ be a homeomorphism such that $g(p_{i_1 i_2  \dots  i_n})=p_{i_2 i_3  \dots  i_n}$, and let $g$ act on $E$ as the shift map on the space $\Sigma_{2}^+$.
Let $X$ be a closed $g$-invariant subset of $E$. Denote $$D_X= \bigcup_{x_{\xi} \in X} [x_{\xi},p]$$ and $$f=g |_{D_X}.$$ 
The proof of the following lemma is clear from the construction, so we omit it.
\begin{lem}\label{lem:subdendrite}
If $X\subset \Sigma^{+}_{n}$ is a subshift then the set $D_X$ is an $f$-invariant subdendrite of the Gehman dendrite $G$.

\end{lem}

\begin{lem}\label{lem:Geh}
If $X$ is a closed and nonempty subset of  $\Sigma_{2}^+$ without isolated points, then the set $D_X= \bigcup_{x_{\xi} \in X} [x_{\xi},p] \subset G$ is homeomorphic with the Gehman dendrite. 
\end{lem}

\begin{proof}
Let $X$ be nonempty subset of $\Sigma_{2}^+$ homeomorphic to the Cantor set. Then the set $D_X= \bigcup_{x_{\xi} \in X} [x_{\xi},p]$ is subdendrite of $G.$ Furthermore, every ramification point of $D_X$ has order 3 and the set of endpoints is homeomorphic with the Cantor set $X$. Using Theorem 4.1 in \cite{are01} such a dendrite is homeomorphic with the Gehman dendrite.
\end{proof}

\section{\texorpdfstring{Uncountable $\omega$-scrambled set without DC3 pairs on dendrite}{omega}}\label{sec:constr}
\begin{lem}\label{lem:Sigma}
There is a Cantor set $\Sigma \subset \Sigma^{+}_{2}$ such that for any $n \geq 2$ and any distinct points $x^{(1)}, x^{(2)}, \dots, x^{(n)} \in \Sigma$ and any $\{i_1, i_2, \dots, i_k\} \subset \{1,2, \dots, n\}$ where $k \in\{1,2, \dots, n\}$, there is $j>0$ such that $$x^{(i)}_j=1 \text{ for } i\in \{i_1, i_2, \dots, i_k\}$$
and $$x^{(i)}_j=0 \text{ for } i \notin \{i_1, i_2, \dots, i_k\}.$$
\end{lem}
\begin{proof}
Fix $n\geq 2$, $1\leq k \leq n$ and indices $\{i_1, i_2, \dots, i_k\} \subset \{1,2, \dots, n\}$. Define  $R_n^{\{i_1, i_2, \dots, i_k\}} \subset \left( \Sigma^{+}_{2}\right)^n$ by 
\begin{align*}
R_n^{\{i_1, i_2, \dots, i_k\}}= \Big\{ &\left(x^{(1)}, x^{(2)}, \dots, x^{(n)}\right)\in \left( \Sigma^{+}_{2}\right)^n \colon  \exists j\in \N 
\text{ such that }
\\&x^{(i)}_j=1 \text{ for } i\in \{i_1, i_2, \dots, i_k\} \text{ and } x^{(i)}_j=0 \text{ for } i \notin \{i_1, i_2, \dots, i_k\} \Big\}. 
\end{align*}
We claim that the set $R_n^{\{i_1, i_2, \dots, i_k\}}$ is open and dense  in the product space $(\Sigma_2^+)^n.$
We endow $(\Sigma_2^+)^n$ with the maximum metric $\rho_n$, i.e. $$\rho_n \left(\left(x^{(1)},\ldots,x^{(n)}\right), \left(y^{(1)},\ldots,y^{(n)}\right)\right)=\max_{i=1,\ldots,n}\rho \left(x^{(i)},y^{(i)}\right).$$
Observe that $R_n^{\{i_1, i_2, \dots, i_k\}}$ is open, since for every $j\in \N$ there is $\eps>0$ such that if $\rho_n \left((x^{(1)},\ldots,x^{(n)}), (y^{(1)},\ldots, y^{(n)})\right)<\eps$
then $x^{(i)}_j=y^{(i)}_j$ for $i=1,\ldots, n$. 
It remains to prove that $R_n^{\{i_1, i_2, \dots, i_k\}}$ is also dense. Fix nonempty words $w_i$ for $i=1,\dots,n$ and put $v^{(i)}=w^{(i)} w^{(i+1)}\ldots w^{(n)} w^{(1)}\ldots w^{(i-1)}.$ Note that $|v^{(1)}|=|v^{(2)}|=\ldots=|v^{(n)}|.$ We define 
$$
u^{(i)}=\begin{cases}
v^{(i)} 1 0^\infty, &\text{ if }i\in \{i_1, i_2, \dots, i_k\},\\
v^{(i)} 0^\infty, &\text{ otherwise.}
\end{cases}
$$
Clearly,
$$
\left(u^{(1)},\ldots, u^{(n)}\right)\in R_n^{\{i_1, i_2, \dots, i_k\}} \cap \left(C[w^{(1)}]\times \dots \times C[w^{(n)}]\right),
$$
which proves that $R_n^{\{i_1, i_2, \dots, i_k\}}$ is dense.
It follows that for each $n$ and $\{i_1,\dots,i_k\}\subseteq \{1,\dots,n\}$ the set $R_n^{\{i_1, i_2, \dots, i_k\}}$ is residual in $(\Sigma_2^+)^n$.
Given $n\geq 2$ let  $\Gamma_n$ be the set of all finite and nonempty subsets of $\{1,\ldots,n\}$.
Then the set 
$$
R_n=\bigcap_{A\in \Gamma_n}R_n^A
$$
is also open and dense, hence is also residual in $(\Sigma_2^+)^n.$ We have constructed a sequence of residual relations, therefore by Mycielski Theorem there is a Mycielski set $M\subset \Sigma_2^+$ which in particular contains a Cantor set $\Sigma \subset M$ such that for any $n\geq 2$ and for any
distinct $x^{(1)},\ldots, x^{(n)}\in \Sigma$ and any $\{i_1, \dots, i_k\} \subseteq \{1, \dots, n\}$ for some $1\leq k \leq n$ we have $\left(x^{(1)},\ldots,x^{(n)}\right)\in R_n^{\{i_1, i_2, \dots, i_k\}}$, that is there is $j> 0$ such that  $x^{(i)}_j=1 \text{ for } i\in \{i_1, i_2, \dots, i_k\}$ and $x^{(i)}_j=0 \text{ for } i \notin \{i_1, i_2, \dots, i_k\}.$
The proof is completed.
\end{proof}
Now we present simple, yet useful lemma which is probably known.
\begin{lem}\label{lem:thick}
For every thick set $P$ there are thick sets $P_i$ such that $P=\bigcup_{i\in \N} P_i$ and $P_i \cap P_j = \emptyset$ provided that $i \neq j$.
\end{lem}
\begin{proof}
For any integer $n\geq 2$ take $j_n$ such that $Q_n=\{j_n, j_n +1,\ldots, j_n+n\}\subset P$. We may assume that $j_{n+1}>j_n+n$.
Put $Q_1=P\setminus \bigcup_{n=2}^\infty Q_n$. Take any bijection $F\colon \N \to \N \times \N$. Let $I_j=\{n \in \N: F(n) \in \{j\} \times \N\}$. For each $j\in \N$ set $P_j=\bigcup_{i\in I_j}Q_i.$ By the construction each $P_j$ is a thick, the sets  $P_j$'s are pairwise disjoint and $P=\bigcup_{i\in \N} P_i$ which completes the proof. 
\end{proof}

\begin{thm}\label{thm:omega2}
For every thick set $P$	there exists a Cantor set $\Gamma \subset \Sigma_P$
such that for any $n\geq 2$, any distinct points $y^{(1)}, y^{(1)}, \dots, y^{(n)} \in \Gamma $  and any choice of indexes ${\{i_1, \dots, i_k\}}\subset \{1,\ldots,n\}$ the set
\begin{equation}
\bigcap_{i\in\{i_1, \dots, i_k\}} \omega_{\sigma}\left(y^{(i)}\right) \setminus \bigcup_{j\notin\{i_1, \dots, i_k\}} \omega_{\sigma}\left(y^{(j)}\right)\label{eq:strongomega}
\end{equation}  contains an uncountable set $D$ without minimal points.
\end{thm}
\begin{proof}
Using Lemma~\ref{lem:thick} we find a decomposition of $P$ into pairwise disjoint thick sets $P=\bigcup_{i\in \N} P_i$.
Let $\Sigma$ be provided by Lemma~\ref{lem:Sigma}. For every $x\in \Sigma$ denote
$$Q_x=\bigcup_{x_n=1}P_n \subset P.$$
Clearly, each $Q_x$ is thick, since every $x\in \Sigma$ contains at least one symbol $1$.
Thus $Q_x$ defines a weakly mixing spacing shift $\Sigma_x =\Sigma_{Q_x}$ (see \cite{lau72} or \cite{ban13}) and so we can select a point $z_x \in \Sigma_x$ with a dense orbit in $\Sigma_x$.
Denote 
$$
\Gamma=\set{z_x : x\in \Sigma}.
$$
Fix any $n\geq 2$ and any $\{i_1, i_2, \dots, i_k\} \subset \{1,2,\dots,n \}$.
Pick any pairwise distinct points $y_1,\ldots, y_n\in \Gamma$ and let $x^{(1)},x^{(2)}, \dots, x^{(n)} \in \Sigma$  be such that $y^{(i)}=z_{x^{(i)}}$.
Then there is $j>0$ such that $x^{(i)}_j=1$ for $i\in \{i_1, i_2, \dots, i_k\}$ and $x^{(i)}_j=0$ for $i \notin \{i_1, i_2, \dots, i_k\}$. This implies that $P_j \subset Q_{x^{(i)}}$ for each $i\in \{i_1, i_2, \dots, i_k\}$ and $P_j \cap Q_{x^{(i)}}=\emptyset$ for each $i \notin \{i_1, i_2, \dots, i_k\}$.
Then 
$$
\Sigma_{P_j} \subset  \bigcap_{i\in\{i_1, \dots, i_k\}} \omega_{\sigma}(z_{x^{(i)}})=\bigcap_{i\in\{i_1, \dots, i_k\}}\Sigma_{x^{(i)}}.
$$
Furthermore, if $i\not\in  \{i_1, \dots, i_k\}$ then in any point $z\in \Sigma_{P_j}\cap \Sigma_{{x^{(i)}}}$ the symbol $1$ occurs at most once. There are at most countably many such points,
and every weakly mixing spacing shift $\Sigma_{P_j}$ is uncountable. Furthermore $P_j$ has thick complement (e.g. contains $P_{j+1}$ in its complement), hence $\Sigma_{P_j}$
is proximal by Theorem 3.11 from \cite{ban13}. But the only minimal point in a proximal system is the fixed point, hence
$$
\bigcap_{i\in\{i_1, \dots, i_k\}} \omega_{\sigma}\left(y^{(i)}\right) \setminus \bigcup_{j\notin\{i_1, \dots, i_k\}} \omega_{\sigma}\left(y^{(i)}\right)
$$
contains an uncountable set $D$ without minimal points.
\end{proof}

Observe that the set $\Gamma$ in Theorem~\ref{thm:omega2} satisfies the following strong version of $\omega$-chaos (in particular is $\omega$-scrambled):
\begin{enumerate}[(i)]
		\item $\left(\omega_{f}\left(x^{\pi(1)}\right)\cap \ldots \cap  \omega_{f}\left(x^{\pi(n-1)}\right)\right)\setminus \omega_{f}\left(x^{\pi(n)}\right)$ contains an uncountable minimal set
		for any permutation $\pi$ of the set $\set{1,\ldots,n}$,
		\item\label{rem:stronger-omega:2} $\bigcap_{i=1}^n \omega_{f}\left(x^{(i)}\right)$ contains an uncountable minimal set.
\end{enumerate}
In fact, conditions posed in \eqref{eq:strongomega} are among the strongest possible dependences between $n$-tuples of $\omega$-limit sets.

Lets take $P=\bigcup_{i=1}^{\infty} \{4^i, 4^i+1, \dots, 4^i+i-1\}$. It is not hard to see that both sets $P$ and $\N\setminus P$
are thick. Then combining Theorem~\ref{thm:omega2} with Corollary 3.19 and Example 3.20 from \cite{ban13} we obtain the following:
\begin{thm}\label{thm:omega3}
There exists a spacing shift $\Sigma_P$ without DC3 pairs  and a Cantor set $\Gamma \subset \Sigma_P$
	such that for any $n\geq 2$, any distinct points $y^{(1)}, y^{(2)}, \dots, y^{(n)} \in \Gamma $  and any choice of indexes ${\{i_1, \dots, i_k\}}\subset \{1,\ldots,n\}$ the set
$$
	\bigcap_{i\in\{i_1, \dots, i_k\}} \omega_{\sigma}\left(y^{(i)}\right) \setminus \bigcup_{j\notin\{i_1, \dots, i_k\}} \omega_{\sigma}\left(y^{(j)}\right)
$$	
contains an uncountable set $D$ without minimal points.
\end{thm}

Lets us recall open question formulated by Ko\v{c}an  in \cite{koc12}.
\begin{enumerate}[(1)]
	\item Does the existence of an uncountable $\omega$-scrambled set imply distributional chaos?
	\item Does the existence of an uncountable $\omega$-scrambled set imply existence of an infinite LY--scrambled set?
	 
	\item Does distributional chaos imply the existence of an infinite LY--scrambled set?
\end{enumerate}

By our construction we immediately obtain a negative answer to the first one i.e. there is an uncountable $\omega-$scrambled set which does not imply distributional chaos.

\begin{cor}\label{cor:omega4}
There exists a continuous self-map $f$ of Gehman dendrite such that:
\begin{enumerate}[(1)]
\item\label{cor:omega4:c1} $f$ does not have DC3 pairs,
\item\label{cor:omega4:c2} $f$ has uncountable $\omega$-scrambled set.
\end{enumerate}
\end{cor}
\begin{proof}
Let $\Sigma_P$ be a spacing subshift provided by Theorem~\ref{thm:omega3}. Taking $X=\Sigma_P$ in  Lemma~\ref{lem:subdendrite} we obtain a map $f\colon D_X\to D_X$ on the Gehman dendrite $D_X$
such that we may view $\Sigma_P\subset D_X$ as a set of endpoints of $D_X$ invariant for $f.$ Furthermore, $f |_{\Sigma_P}$ is topologically conjugate to $\sigma$ on $\Sigma_P.$ 
Moreover, there is a fixed point $p\in D_X\setminus \Sigma_P$ and for every $y\in D_X\setminus \Sigma_P$ there is $n>0$ such that $f^n(y)=p$.

It is clear that \eqref{cor:omega4:c2} is satisfied, since $(\Sigma_P,\sigma)$ is a subsystem of $(D_X,f)$ containing an uncountable $\omega$-scrambled set.
Furthermore, if pairs $(z_1,y_1)$, $(z_2,y_2)$ are asymptotic then $F_{z_1z_2}=F_{y_1y_2}$ and $F^*_{z_1z_2}=F^*_{y_1y_2}$ (e.g. see \cite[Lemma~5]{malopr11}). This shows that there is no DC3 pair in $D_X\setminus \set{p}$. But distance in $D_X$ is given by arclength, hence we may assume that $\rho(p,x)=1$ for every $x\in \Sigma_P$. This implies that
$(p,x)$ is not DC3 for any $x\in \Sigma_P$ and automatically $(p,y)$ is not DC3 for any $y\in D_X\setminus \Sigma_P$. The proof is finished.
\end{proof}

\section{Shift without DC3 pairs extended to mixing} \label{sec:mix}

Let $X$ be a shift. Note that $\rho(x,y)<2^{-k}$ for some $k \geq 0$ and $x,y \in X$ implies that $\rho(x,y)\leq 2^{-k-1},$ hence $\rho(x,y)<t$ for every $t \in (2^{-k-1},2^{-k}).$  
It follows that $F_{xy}^{(n)}(t)=F_{xy}^{(n)}(2^{-k})$ for each $n \geq 1$, $k \geq 0$, $x,y \in X$ and $t \in (2^{-k-1},2^{-k}).$ Thus $F_{xy}(t)=F_{xy}(2^{-k})$ and $F_{xy}^*(t)=F_{xy}^*(2^{-k})$ for  $t \in (2^{-k-1},2^{-k}).$ In other words, $F_{xy}$ and $F_{xy}^*$ are piecewise constant functions.

\begin{lem}\label{lem:shift}
Let $X$ be a shift. Then  $(x,y) \in X \times X$ is a  DC3 pair if and only if there exists an integer $k\geq 0$ such that $F_{xy}(2^{-k})<F_{xy}^*(2^{-k}).$ 
 \end{lem}
 \begin{proof}
Assume first that $F_{xy}(2^{-k})<F_{xy}^*(2^{-k})$ for some integer $k \geq 0.$  
It follows that $ F_{xy}(t) < F^*_{xy}(t)$ for $t \in (2^{-k-1},2^{-k})$ so $(x,y)$ is a  DC3 pair.

On the other hand, if $(x,y)$ is a  DC3 pair then there exists an interval $(a,b)\subset [0,1]$ such that  $F_{xy}(t)<F_{xy}^*(t)$ for all $t \in (a,b).$ Let $k \geq 0$  be such that  $2^{-k-1} <b \leq2^{-k}.$ But $F_{xy}$ and $F_{xy}^*$ are constant over each interval $(2^{-k},2^{-k+1})$ hence $F_{xy}(2^{-k}) <F_{xy}^*(2^{-k})$ for some $k$.
\end{proof}

\begin{lem} \label{lem:dc3}
If $X \subset \Sigma_2^+$ is a shift space such that for every $x \in X$ we have $d\left(\{i:x_i=0\}\right)=1$ then there is no DC3 pair in $X.$
\end{lem}
\begin{proof}
Let $x,y \in X$. Then 
\begin{align*}
\bar{d}\left(\{i:x_i\neq y_i\}\right)& = \bar{d}\left(\{i:x_i\neq 0, y_i=0\} \vee \{i:x_i =0, y_i \neq 0\}\right)  \\
& \leq \bar{d}\left(\{i:x_i\neq 0, y_i=0\}\right) + \bar{d}\left(\{i:x_i= 0, y_i \neq 0\}\right)\\
&\leq d\left(\{i:x_i\neq0\}\right)+d\left(\{i:y_i\neq0\}\right)=0.
\end{align*}
Thus $d(\{i:x_i=y_i\})=1$ which clearly implies that $(x,y)$ is not a DC3 pair.
\end{proof}
\begin{rem}
Let $X$ be a shift space. Then $d\left(\{i:x_i=0\}\right)=1$ for every $x \in X$ if and only if the measure concentrated on $0^{\infty}$ is the only invariant measure for $X$  (see \cite{kwie13}).
\end{rem}

\begin{lem}\label{lem:infty}
Let $X$ be a shift and $x \in X$ such that $d(\{i: x_i=0\})=1$. If $(x,y)$  is a DC3 pair then $(0^\infty, y)$ is a DC3 pair.
\end{lem}

\begin{proof}
By Lemma~\ref{lem:shift} it is enough to check that $F_{xy}(2^{-k})<F_{xy}^*(2^{-k})$ for some $k\geq 0.$

If $(x,y)$ is a DC3 pair then there exist an integer $l>0,$ an increasing sequence $(s_i)_{i=1}^{\infty}$ of positive integers  and  $\gamma>0$ such that
$$\frac{1}{s_i}\left| \left\{ 0< k <s_i: x_{[k,k+l)} \neq y_{[k,k+l)}\right\}\right| \geq \gamma.
$$
Observe that we can write the set
$ \left\{ 0< k <s_i: x_{[k,k+l)} \neq y_{[k,k+l)}\right\}$  as the disjoint union of 
$$\left\{ 0< k <s_i: x_{[k,k+l)} \neq y_{[k,k+l)} \text{ and } x_{[k,k+l)}\neq 0^l \right\}$$ and 
$$\left\{ 0< k <s_i: y_{[k,k+l)} \neq x_{[k,k+l)}=0^l \right\}.$$
Now, $d\left(\{i: x_i=0\}\right)=1$ implies that for every $k>0$ we have  $$d\left(\{i: x_{[i,i+k)}=0^k\}\right)=1.$$
We get that for every $\delta>0$ there exists $N$ such that for all $i>N$ we have
$$ \frac{1}{s_i}\left|\left\{ 0< k <s_i: x_{[k,k+l)} \neq y_{[k,k+l)} \text{ and } x_{[k,k+l)}\neq 0^l \right\} \right| <\delta.$$
From that we obtain
\begin{align*}
\gamma \leq \frac{1}{s_i}\left| \left\{ 0< k <s_i: x_{[k,k+l)} \neq y_{[k,k+l)}\right\}\right| \leq \frac{1}{s_i}\left|\left\{ 0< k <s_i: y_{[k,k+l)} \neq 0^l \right\} \right|+ \delta
\end{align*}
and finally
$$ \frac{1}{s_i} \left|\left\{ 0< k <s_i: y_{[k,k+l)} \neq 0^l \right\} \right| \geq \gamma - \delta.$$
On the other hand there exist a decreasing sequence $(t_i)_{i=1}^{\infty}$ and $0<\alpha<\gamma$ 
$$\frac{1}{t_i}\left| \left\{ 0< k <t_i: x_{[k,k+l)} \neq y_{[k,k+l)}\right\}\right| \leq \alpha.$$
A similar calculations as above gives that
$$ \frac{1}{t_i}\left| \left\{ 0< k <t_i: y_{[k,k+l)}\neq 0^l\right\}\right| \leq \alpha - \eps
$$
where $\eps>0$ is such that 
$$ \frac{1}{t_i}\left|\left\{ 0\leq k <t_i: x_{[k,k+l)} \neq y_{[k,k+l)} \text{ and } x_{[k,k+l)}\neq 0^l \right\} \right| <\eps.$$
Since $\eps$ and $\delta$ can be arbitrarily small this completes the proof that $(0^\infty, y)$ is a DC3 pair.
\end{proof}

\begin{thm}\label{thm:mix}
Let $X$ be a shift such that for every $x \in X$ we have $d(\{i: x_i=0\})=1.$ Then there exists a mixing shift $Y$ containing $X$  
such that there is no  DC3 pair in $Y$.
\end{thm}

\begin{proof}
First, note that from Lemma~\ref{lem:dc3} we have that there in no DC3 pair in X.
Now we inductively construct an increasing sequence of shift spaces $X_0 \subset X_1 \subset \dots$ and then define space $Y$ as the closure of the union of all $X_n$'s i.e.
$$Y= \overline{\bigcup_{n=0}^{\infty} X_n}.$$
Let $X_0=X \cup W,$ where  $W$ is any shift such that $\# W >0$ and $d(\{i \colon x_i=0\})=1$ for any $x \in W$.
For example  see  \cite[Example~8.5]{kul14}. We define the set $X_1$ adding to $X_0$ orbits of points of the form 
$$0^\alpha 1 0^\beta 1 0^\infty, \text{ where } \alpha \geq 0 \text{, } \beta \geq 2.$$
Note that every block added at first step has at most two occurrences of the symbol 1.

By  $\varphi_n,$ $\varphi_n^l$ we denote the maximum number of occurrences of the symbol 1 among all blocks of length $n$ in $X$ and $X_l$, respectively, that is, 
$$\varphi_n = \max \{\|w\|_1: w\in \mathcal{L}_n (X_0)\} $$ and $$\varphi_n^l =  \max \{\|w\|_1: w\in \mathcal{L}_n (X_l)\},$$ provided that $\varphi^l_n$ was already defined.

Now, inductively, for given $X_n$ and $n \geq 0$ we construct a shift space $J_{n+1}$ and set $X_{n+1}=\overline{X_n \cup J_{n+1}}$, where
$$J_{n+1}=\bigcup_{m=0}^{\infty} \sigma^m \left(\{ 0^\alpha u 0^\beta v 0^\infty: \alpha \geq 0\text{, } u,v \in \mathcal{L}_n(X_n)\text{, } \beta \text{ such that } \varphi_{\beta} >2 \varphi^n_{n}\}\right) $$
Note that each $X_n$ is a subshift and let us first notice that the sequence $(\varphi_n)_{n=1}^{\infty}$ is non-negative and subadditive (i.e. $0\leq \varphi_{m+n}\leq \varphi_m + \varphi_n$).

Observe that $$J_{n+1}=\overline{J_{n+1}}$$ so $$X_{n+1}=X_n \cup J_{n+1}.$$ We claim that for all $l \in \N$ and $n \in \N$ $$\varphi_n=\varphi_n^l.$$
Let us notice first that from property that $X \subset X_l$ we have that $\varphi_n \leq \varphi^l_n$ for all $l \in \N.$ The inequality $\varphi_n^l \leq \varphi_n$ we prove by induction on $l\in \N$. It is clear when $l=0$. Now, assume that $\varphi_n^{l-1} \leq \varphi_n$ for some $l > 0.$ Let $u\in \mathcal{L}_n(X_l).$ We get two possibilities, $u \in \mathcal{L}_n(X_{l-1})$ or $u \in \mathcal{L}_n(J_l).$ 
If $u \in L_n(X_{l-1})$ we have
$$\|u\|_1 \leq \varphi^{l-1}_n \leq \varphi_n.$$
On the other hand if $u \in \mathcal{L}_n(J_l)$ and  $\beta$ from the definition of $J_l$ is fixed for some $x \in J_l$ containing $u$ as a subword, we get two cases.

{\it Case 1.} If $|u| \leq \beta$ we have that $u=0^kw0^s$ for some $k, s \geq 0$ and $w \in \mathcal{L}_r(X_{l-1})$ for $r<n.$ Then
$$\|u\|_1 \leq \varphi_r^{l-1}\leq \varphi_n^{l-1}\leq\varphi_n.$$

{\it Case 2.} Suppose that $|u| > \beta$. Now we get that $u \sqsubset 0^sw0^{\beta}v0^t$ for $s, t \geq 0$ and $w,v \in \mathcal{L}_{l-1}(X_{l-1})$. Therefore by definition of $\beta$ we obtain
$$\|u\|_1 \leq 2\max\left\{\|v\|_1 \colon v \in \mathcal{L}_{l-1}(X_{l-1})\right\} \leq 2 \varphi^{l-1}_{l-1}\leq\varphi_{\beta}<\varphi_n.$$ The claim is proved.

Let us notice that the condition $\lim_{n \to \infty}\frac{\varphi_n}{n}=0$ is equivalent to $d(\{i: x_i = 0\})=1$ (see \cite[Theorem 3]{kwifal15}).
 
Now we will prove that $d(\{i: y_i = 1\})=0$ for every $y \in Y.$ If $y \in \bigcup^{\infty}_{n=0} X_n$ then $y \in X_k$ for some $k$ and then $$\frac{1}{n} \left|\{0\leq i <n: y_i=1\} \right| \leq \frac{1}{n}\cdot \varphi_n^k \leq \frac{\varphi_n}{n} \to 0$$
Now fix $y \in Y \setminus \bigcup^{\infty}_{n=0} X_n.$ Then there exists a sequence $x_k \in X_{n_k}$ where $(n_k)_{k=0}^{\infty} \subset \N$ is an increasing sequence and 
$$\lim_{k \to \infty}x_k=y \in Y \setminus \bigcup^{\infty}_{n=0} X_n.$$ 
For each $n$ there exist $N \geq 0$ such that for all $k \geq N$
$$y_{[0,n)}=(x_k)_{[0,n)}.$$ 
In particular, for some sufficiently large $l$ we get
that
$$\left| 0\leq i \leq n \colon \varphi_i=1 \right| = \left|\{0\leq i <n: (x_l)_i=1\} \right|  \leq \varphi_n^l \leq \varphi_n.$$

The shift space $Y$ is topologically mixing.
Indeed, fix any words $u, v \in \mathcal{L}(Y).$ We get that there exists $k\geq 0$ such that $u, v \in \mathcal{L}(X_k)$. 
Extending $u$ and $v$ to possibly larger words we may assume that $u,v \in \mathcal{L}_m(X_k)$ for some $m \geq k.$ But $X_k \subset X_{k+1} \subset \dots \subset X_m$ so $u,v \in \mathcal{L}_m(X_m).$

Therefore $u0^\alpha v 0^\infty \in J_{k+1} \subset Y$ for all $\alpha \geq 0$ sufficiently large. It ends the proof that $Y$ is topologically mixing.
\end{proof}

\section{Distributional chaos without an infinite LY-scrambled set} \label{sec:inf}
Denote $I=[0,1].$ We perform an inductive construction. For the initial step set $m_0=1$ and $z_1^{(0)}=\frac{1}{2}.$ Define $Z^{(0)}=\left\{z_1^{(0)}\right\},$  $x_0^{(0)}=0,$ $x_1^{(0)}=z_1^{(0)},$ $x_2^{(0)}=1,$ and $l_0=1, l_{-1}=0.$  Note that the sequence $\left(x_i^{(0)}\right)_{i=0}^{2}$ contains $Z^{(0)},$ constructed so far, and the endpoints of $[0,1].$ Furthermore $x_0^{(0)}<x_1^{(0)}<x_2^{(0)}.$ 
For $i \in \N$ denote $$L_i=(l_{i-1}+1)m_i$$ with $$m_{i+1} \geq 2^i l_i.$$
Now for the inductive step assume that we have just constructed sets $$Z^{(0)}, Z^{(1)}, \dots, Z^{(n)} \subset I$$ where $$\left|Z^{(i)}\right| = 
L_{i}$$ and 
  $$l_i=\sum_{j=0}^i \left|Z^{(j)}\right|$$ for $i\leq n.$
In particular $Z^{(i)} \cap Z^{(j)} = \emptyset$ for $i \neq j.$ Note that $l_i$ and $L_i$ are constructed in such a way that $l_i+1\geq 2^i$ and $L_i \geq 2^i$ for all $i \in \N_0.$\\
We also assume that all elements of set $$\bigcup_{p=0}^n Z^{(p)}=\left(x_j^{(n)}\right)_{j=1}^{l_n} 
$$  were enumerated in such a way that 
$$0=x_0^{(n)}<x_1^{(n)}<\dots<x_{l_{n}+1}^{(n)}=1.$$
Additionally we assume that sets $Z^{(i)}=\left\{z_1^{(i)}, \dots, z_{L_i}^{(i)}\right\}$ are such that if we put $z_0^{(i)}=0$ and $z_{L_{i}+1}^{(i)} = 1$ then $\left| Z^{(i+1)}\cap \left(x_j^{(i)}, x_{j+1}^{(i)}\right)\right|=m_{i+1}$ for all $i \in \N_0$ and $j \in \{0,1, \dots, L_{i}\}.$ 
\\
Now fix any $$m_{n+1} \geq 2^n l_n $$ and define 
\begin{equation}\label{eq:defz}
z^{(n+1)}_{i \cdot m_{n+1}+k}= x_i^{(n)}+\frac{k}{m_{n+1}+1} \left( x_{i+1}^{(n)}-x_i^{(n)}\right) 
\end{equation}
for $i=0,1, \dots, l_n,$ $k=1,2,\dots, m_{n+1}.$ Then put 
$$Z^{(n+1)}=\left\{z_j^{(n+1)}: j=1,2, \dots, L_{n+1}\right\}.$$ Finally enumerate elements of set
$$\bigcup_{p=0}^{n+1} Z^{(p)}=\left(x_j^{(n+1)}\right)_{j=1}^{l_{n+1}}$$
in such a way that
$$0=x_0^{(n+1)}<x_1^{(n+1)}<\dots<x_{l_{n+1}+1}^{(n+1)}=1.$$
Now we define the sequence of points $$A_{n,k} = \left(z_k^{(n)}, \frac{1}{2^n}\right) \in I \times I$$ for $n \in \N_0 $ and $k \in \{1,2, \dots, L_{n}\}.$ 

Let us notice that by our construction  
\begin{equation}\label{eq:diamz}
\diam \left(z_j^{(k+1)}, z_{j+1}^{(k+1)}\right) \leq \frac{1}{m_{k+1}+1} \diam\left(z_j^{(k)}, z_{j+1}^{(k)}\right) \end{equation}
 for any $k$ and $j$ (see (\ref{eq:defz})) . 

If $\pi$ denotes projection on the first coordinate, i.e. $\pi(x,y)=(x,0),$ we get that $\pi(A_{n,i}) \neq \pi(A_{m,j})$ for every $n \neq m$ and every $i \in \{1,2,\dots, L_n\},$  $j\in \{1,2,\dots, L_m\}.$

\begin{figure}[htp]
	\centering
		\includegraphics[width=0.5\linewidth]{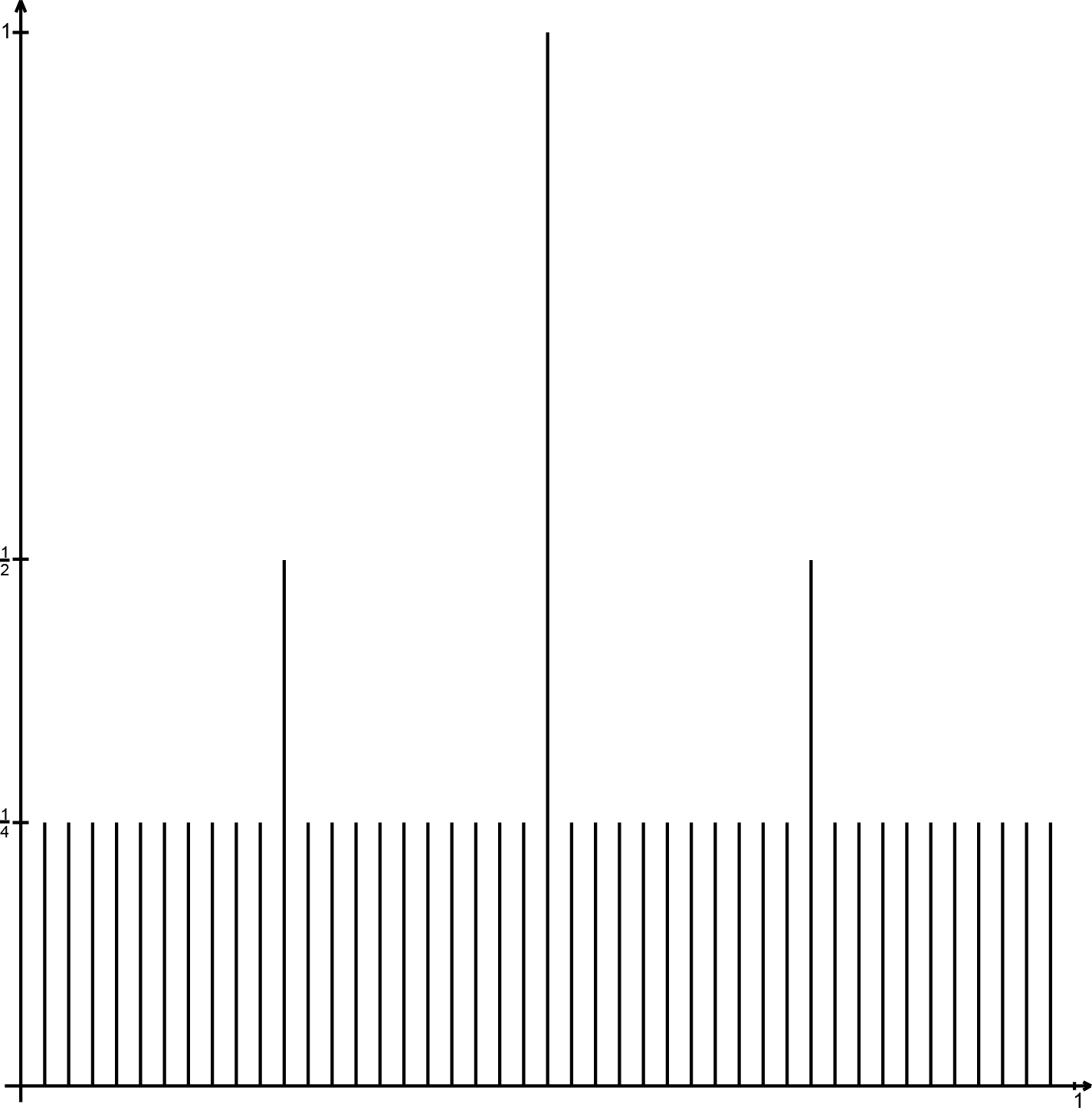}
	\caption{The result of the first three steps of the construction of the dendrite $\mathcal{D}$.}
	\label{fig:D}
\end{figure}

\begin{lem}
The set
$$\mathcal{D}=\left(I \times \{0\} \right)\cup \bigcup_{n=0}^{\infty} \bigcup_{k=1}^{ L_n} \left\{ \left(z_k^{(n)},y\right): y \in \left[0, \frac{1}{2^n} \right] \right\}$$
is a dendrite.
\end{lem}
\begin{proof}
It is enough to show that $\mathcal{D}$ is locally connected. 
Let $x \in \mathcal{D}$ and let $U$ be a non-empty open set such that $x\in U$. Let consider two cases.
First if $x\in \mathcal{D}\setminus \left(I \times \{0\}\right).$ 
There exist some $t \geq 0$ such that $x=(a,b) \in \mathcal{D} \cap I \times \left[\frac{1}{2^{t+1}}, \frac{1}{2^t}\right].$
Let denote $$\eps_1 = \min \left\{ \left|b-\frac{1}{2^{t-1}}\right|, \left|b-\frac{1}{2^{t+2}}\right| \right\}$$
and 
$$\delta_t = \min_{n \leq t+2}  \min_{1\leq l \leq L_n }  {\left\{ \left|a-z^{(n)}_l\right|: a \neq z^{(n)}_l \right\} }
.$$
Now taking $\eps = \frac{1}{2} \min \{\eps_1, \delta_t\}$ the set $$V=((a-\eps, a+\eps) \times (b-\eps, b+\eps)) \cap \mathcal{D}$$ is connected because it is a segment $\{a\} \times (b-\eps, b+\eps).$
Finally if $x=(x_0,0) \in I\times \{0\}$ then taking neighborhood $(x_0-\eps_0, x_0+\eps_0) \times [0,\eps_0)$ for some $\eps_0 >0$ we get connected set $V\subset U.$ It ends the proof.
\end{proof}
For $n\in N_0$ and $k\in \{1, \dots, L_n\} \subset \N$ denote by $\phi_n, \psi_{n,k}$ and $\widehat{\psi}_{n,k}$ increasing linear functions such that 
$$\phi_n\left(\left[\frac{3}{2^{n+2}},\frac{1}{2^{n}}\right]\right)=\left[0, \frac{1}{2^n}\right],$$

$$\psi_{n,k}\left(\left[\frac{1}{2^{n+1}},\frac{3}{2^{n+2}}\right]\right)=
\begin{cases}
\left[z_{k}^{(n)},z_{k+1}^{(n)}\right]  \text{ if $k \neq L_n,$}\\
\\
\left[z_{L_{n+1}-1}^{(n+1)},z_{L_{n+1}}^{(n+1)}\right]  \text{ if $k=L_n,$}\\
\end{cases}
$$

$$\widehat{\psi}_{n,k}\left(\left[\frac{1}{2^{n+1}},\frac{3}{2^{n+2}}\right]\right)=
\begin{cases}
\left[z_{L_n-k}^{(n)},z_{L_n+1-k}^{(n)}\right]  \text{ if $k \neq L_n,$}\\
\\
\left[z_{1}^{(n+1)},z_{2}^{(n+1)}\right]  \text{ if $k=L_n.$}\\
\end{cases}
$$
Now we  define a map $f: \mathcal{D} \to \mathcal{D}.$ Given $(x,y) \in \mathcal{D}$ we consider the following cases:
\begin {enumerate}[(1)]
\item if $x=z_k^{(n)},  y\in \left[\frac{3}{2^{n+2}},\frac{1}{2^{n}}\right]$ for some $n\in \N_0,$ $k\in \{1, \dots, L_n\} \subset \N$ and $n$ is even then
$$f(x,y)=f\left(\left(z_k^{(n)},y\right)\right)=
\begin{cases}
\left(z_{k+1}^{(n)}, \phi_n (y) \right) \text{ if $k \neq L_n$}\\
\\
\left(z_{L_{n+1}}^{(n+1)},\frac{1}{2} \phi_n (y)\right) \text{ if $k=L_n$},\\
\end{cases}
$$

\item if $x=z_{L_n+1-k}^{(n)}$ and $y\in \left[\frac{3}{2^{n+2}},\frac{1}{2^{n}}\right]$ for some $n\in \N_0,$ 
$k \in \{1, \dots, L_n\}\subset \N$  and  $n$ is odd then
$$f(x,y)=f\left(\left(z_{L_n+1-k}^{(n)},y\right)\right)=
\begin{cases}
\left(z_{L_n-k}^{(n)}, \phi_n (y) \right) \text{ if  $k \neq L_n,$}\\
\\
\left(z_{1}^{(n+1)},\frac{1}{2} \phi_n (y)\right) \text{ if  $k =L_n,$} \\
\end{cases}
$$

\item 
if $x=z_k^{(n)}$ and  $ y\in \left[\frac{1}{2^{n+1}},\frac{3}{2^{n+2}}\right]$ for some $n\in \N_0$  and $k\in \{1, \dots, L_n\}\subset \N$ then
$$f(x,y)=f\left(\left(z_k^{(n)},y\right)\right)=
\begin{cases}
\left(\psi_{n,k}(y),0\right) \text{ if $n$ is even},\\
\\
\left(\widehat{\psi}_{n,k}(y),0\right) \text{ if $n$ is odd},\\
\end{cases}
$$
\item if $x=z_k^{(n)}$ and $y\in \left[0,\frac{1}{2^{n+1}}\right]$ for some $n \in \N_0$ and $k\in \{1, \dots, L_n\}\subset \N$ then
 $$f(x,y)=f\left(\left(z_k^{(n)},y\right)\right)=\pi\left(A_{n,k}\right),$$
\item if $(x,y) \in I \times \{0\}$ then $$f(x,y)=(x,0).$$
\end{enumerate}
By $I^{n}_{k}$ we denote the segment connecting points $\left(z_k^{(n)},0\right)$ and $A_{n,k}.$ Note that $$\diam I_k^n=\frac{1}{2^n}$$ for every $k\in \{1,2,\dots, L_n\}.$
\begin{figure}[htp]
	\centering
	\includegraphics[width=0.5\linewidth]{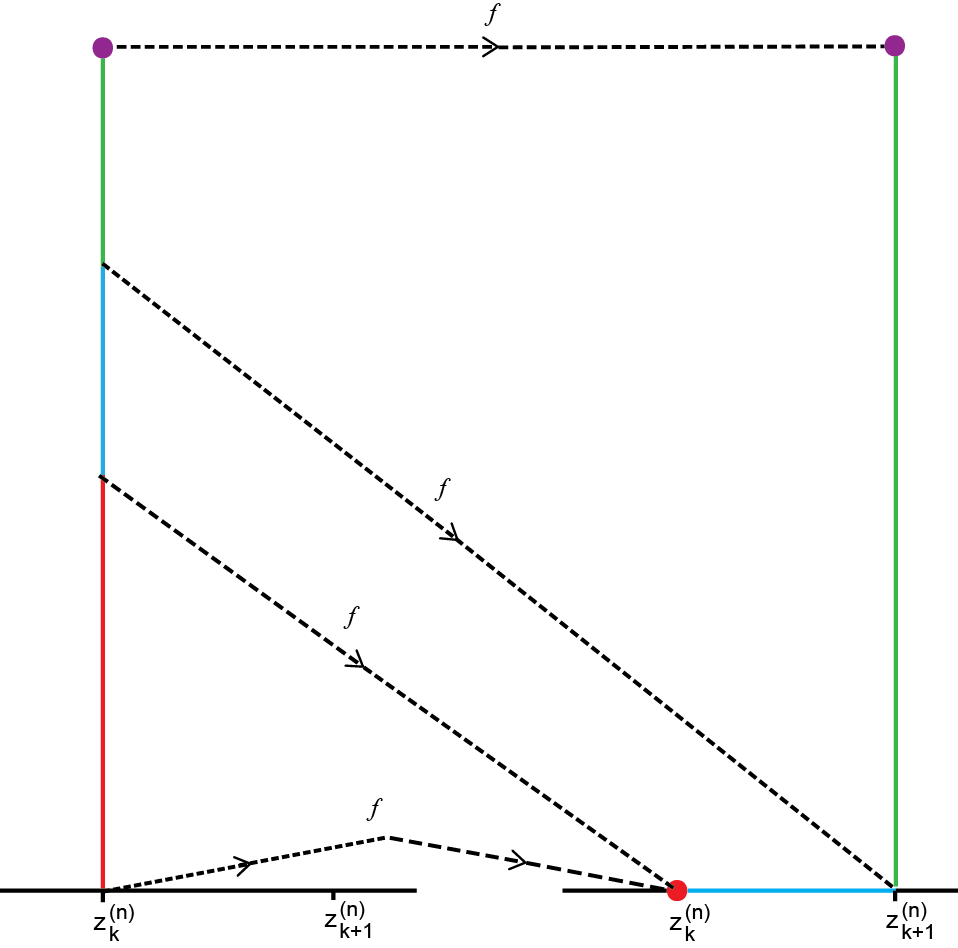}
	\caption{Sketch how the map $f$ acts}
	\label{fig:funF}
\end{figure}

\begin{lem}
The map $f$ is a continuous map on $\mathcal{D}$.
\end{lem}
\begin{proof}
Let $y \in\mathcal{D}$ and $\left(y_n\right)_{n=1}^{\infty} \subset  \mathcal{D}$ be the sequence such that $$\lim_{n \to \infty} y_n = y.$$
Let us consider the cases.\\
\textit{Case 1. }
Let $y \in \mathcal{D} \setminus \left(I \times \{0\}\right)$ i.e. $y\in  I_k^m$ for some $m\in \N$ and $k\in \{1,2,\dots, L_m\}.$ 
\begin{enumerate} [({1}a)]
\item If we have that $\rho \left(y_n, \pi \left(y_n\right)\right) > \frac{1}{2^{m+1}}$ for all $n$ then 
$y_n\in \bigcup_{r\leq m} \bigcup_{j \leq L_r} I^r_j.$
Since there are finitely many $I_j^r$ we may assume that $y_n \in I^r_j$ for some fixed $j,r$ and all $n.$ But then $r=m$ and $j=k$ and by continuity of functions $\phi_m, \psi_m$ and $\widehat{\psi}_m$ we get that $f$ is continuous.
\item If $\rho(y, \pi(y))\leq \frac{1}{2^{m+1}} .$ It means that there exists $t> m$ such that $y=(y_1,y_2) \in I^{m}_k \cap \left(I \times \left[\frac{1}{2^{t+1}}, \frac{1}{2^{t}}\right]\right).$
Let us denote 
$$\eps_t = \min_{n \leq t+2}  \min_{1\leq l \leq L_n } 
 {\left\{ \left|y_1-z^{(n)}_l\right|: y_1\neq z^{(n)}_l \right\} } 
$$
and 
$$\delta_j=\left|y_2-\frac{1}{2^{j}}\right|.$$
Now take
$$\eps = \min \left\{ \frac{\eps_t}{2},\frac{\delta_{t-1}}{2}, \frac{\delta_{t+2}}{2} \right\}$$
and assume that $\rho(y_n, y) <\eps$.
We get that there exists $N \in \N$ such that for all $n>N$  we have  $y_n \in I^m_k$ and $$\diam \left[y_n, \pi(y_n)\right] <\frac{1}{2^{m+1}}$$
and therefore
$$\rho(f(y_n),f(y))<\rho(y_n, y) \to 0$$
which proves that $f$ is continuous.

\end{enumerate}
\textit{Case 2. }
Let $y \in I \times \{0\}$ and $\rho \left(y_n, \pi \left(y_n\right)\right) < \frac{1}{2^{m+1}}$ for all $n.$
\begin{enumerate}[({2}a)]
\item If $f(y_n)=\pi(y_n)$  for all $n$  
then
 $$\rho(f(y_n),f(y)) \leq \rho(\pi(y_n), y) < \rho(y_n, y) \to 0$$
 so the $f$ is continuous.
 \item If $f(y_n) \neq \pi(y_n)$ for all $n$ then $y_n\in I^r_j$ for some $r>0$ and $j \in \{1,2,\dots, L_r\}$ and $\diam I^r_j < \frac{1}{2^m}$. 
 It means that for even $n$ we have $$f(y_n) \in I^r_{j+ 1}$$ or 
$$f(y_n) \in \left[z_j^{(r)}, z_{j+1}^{(r)}\right]$$ or
$$f(y_n) \in I^{r+1}_{L_{r+1}}  \text{ for } j=L_r$$ or
$$f(y_n) \in \left[z_{L_r}^{(r)}, z_{L_{r+1}}^{(r+1)}\right]$$
 and by the other hand if $r$ is odd we have
$$f(y_n) \in I^r_{j-1}$$ or
$$f(y_n) \in \left[z_{j-1}^{(r)}, z_{j}^{(r)}\right]$$  or
$$f(y_n) \in I^{r+1}_1 \text{ for } j=1$$ or
$$f(y_n) \in \left[z_1^{(r+1)}, z_{1}^{(r)}\right].$$ 
Now we get 
 $$\rho(f(y_n),f(y)) \leq \rho(f(y_n),y_n)+\rho(y_n,y).$$
If $f(y_n) \in I^r_{j\pm 1}$ then
$$\rho(f(y_n),y_n) \leq \frac{2}{2^r}+\frac{1}{l_r+1} \leq \frac{3}{2^r} <\frac{1}{2^{r-2}}.$$ 
and if $f(y_n) \in I^{r+1}_{L_{r+1}}$  or $f(y_n) \in I^{r+1}_{1}$ we get that
$$\rho(f(y_n),y_n)\leq \frac{1}{2^r}+\frac{1}{l_r+1} +\frac{1}{2^{r+1}}\leq \frac{3}{2^r}< \frac{1}{2^{r-2}} .$$
If $f(y_n) \in I \times  \{0\}$ we get
$$\rho(f(y_n),y_n)\leq \frac{1}{2^r}+\frac{1}{l_r+1} \leq \frac{1}{2^{r-1}} <\frac{1}{2^{r-2}}.$$
Now we have that
 $$\rho(f(y_n),f(y)) \leq \rho(f(y_n),y_n)+\rho(y_n,y) < \frac{1}{2^{r-2}}+\rho(y_n,y) \to 0$$
because if $n$ grows then $r$ also grows. 
Indeed,  by contradiction suppose that there exists $r(n)$ and subsequence $(y_{n_k})_{k=1}^{ \infty}$  such that $y_{n_k} \in I^r_j$ for all $k$. We get that $y \in I^r_j$ and $$\rho(f(y_{n_k}), \pi (y_{n_k})) <\frac{1}{2^{r+1}}$$ and it implies that $f(y_{n_k})=\pi(y_{n_k})$ which is contradiction with assumptions that $f(y_n) \neq \pi(y_n)$ in this case. So, the proof of continuity of the map $f$ is complete.

\end{enumerate}

\end{proof}
\begin{lem} \label{lem:dis}
Let $x,y \in \mathcal{D}$. If for every positive integer $n$ points $f^n(x),f^n(y)$ are not fixed points of $f$ then $$\lim_{n \to \infty} \rho(f^n(x), f^n(y))=0.$$ 

\end{lem}
\begin{proof}
Without loss of generality we may assume that there are $m, i \geq 0$ and $k \geq 0$ such that $y\in I^m_i$ and $f^k(x)\in I^m_i.$
Let $\eps>0$. Then there exist $s$ and $m^{\prime} $ such that $\frac{1}{2^{m^{\prime}-1}}+\frac{k}{l_{m^{\prime} }+1}<\eps$ and
$$f^s(y) \in  I^{m^ {\prime}}_{i+k},$$

$$f^s(x) \in  I^{m^{\prime}}_{i}$$
and 
$$\dist\left(I^{m^ {\prime}}_{i+k}, I^{m^{\prime} }_{i}\right)<  \frac{k}{l_{m^{\prime} }+1}$$
for $i , i+k \in \{1,2, \dots, L_{m^{\prime} }\}.$
Now we have 
$$\rho\left(f^s(x), f^s(y)\right) \leq \diam{I^{m^ {\prime}}_{i+k}} + \dist\left(I^{m^ {\prime}}_{i+k}, I^{m^{\prime} }_{i}\right) + \diam{I^{m^{\prime} }_{i}}  \leq \frac{1}{2^{m^{\prime}-1}}+\frac{k}{l_{m^{\prime} }+1}<\eps.$$

By definition there is $M>0$ such that for $r \geq s+M$ we have  $f^r(x) \in I_j^{m^{\prime \prime}}$ for some $j$ and $m^{\prime \prime} > m^{\prime}.$ Let us consider the cases.

\begin{enumerate}[(a)]
\item If $r$ is even and $f^r(y) \in I_{j+k}^{m^{\prime \prime}}$ for $j+k \leq L_{m^{\prime \prime}}$ or if $r$ is odd and $f^r(y) \in I^{m^{\prime \prime}}_{j-k}$ for $j-k \geq 1$
we get that
\begin{align*}
\rho\left(f^r(x), f^r(y)\right) & \leq \diam{I^{m^{\prime \prime}}_{j \pm k}} + \dist\left(I^{m^{\prime \prime}}_{j \pm k}, I^{m^{\prime \prime}}_{j}\right) + \diam{I^{m^{\prime \prime}}_{j}}  \leq \\& 
\leq \frac{1}{2^{m^{\prime \prime}-1}}+\frac{k}{l_{m^{\prime \prime}}+1} \leq \frac{1}{2^{m^{\prime}-1}}+\frac{k}{l_{m^{\prime} }+1}<\eps.
\end{align*}

\item If $j+k > L_{m^{\prime \prime}}$ and $r$ is even then $f^r(y) \in I_{L_{m^{\prime \prime}+1}+L_{m^{\prime \prime}}-j-k}^{m^{\prime \prime}+1}$ and we have

\begin{align*}
\rho\left(f^r(x), f^r(y)\right) & \leq \diam{I_{L_{m^{\prime \prime}+1}+L_{m^{\prime \prime}}-j-k}^{m^{\prime \prime}+1}} + 
\dist\left(I_{L_{m^{\prime \prime}+1}+L_{m^{\prime \prime}}-j-k}^{m^{\prime \prime}+1}, I^{m^{\prime \prime}}_{j}\right) 
+ \diam{I^{m^{\prime \prime}}_{j}} \leq 
\\&
\leq \frac{1}{2^{m^{\prime \prime}+1}}+
\dist\left(I_{L_{m^{\prime \prime}+1}+L_{m^{\prime \prime}}-j-k}^{m^{\prime \prime}+1}, I^{m^{\prime \prime}}_{L_{m^{\prime \prime}}}\right) 
+
\dist\left(I_{L_{m^{\prime \prime}}}^{m^{\prime \prime}}, I^{m^{\prime \prime}}_{j}\right) 
+
\frac{1}{2^{m^{\prime \prime}}} \leq
\\&
\leq \frac{1}{2^{m^{\prime \prime} +1}}+\frac{j+k-L_{m^{\prime \prime}}}{l_{m^{\prime \prime}}+1}+\frac{L_{m^{\prime \prime}}-j}{l_{m^{\prime \prime}}+1}+\frac{1}{2^{m^{\prime \prime}}} \leq
\\&
\leq \frac{1}{2^{m^{\prime \prime}-1}}+\frac{k}{l_{m^{\prime \prime}}+1} \leq \frac{1}{2^{m^{\prime}-1}}+\frac{k}{l_{m^{\prime} }+1}<\eps.
\end{align*}
\item  If $f^r(y) \in I^{m^{\prime \prime}-1}_{k-j}$ for $j-k<1$ and $r$ is odd then in the similar way as above we get

\begin{align*}
\rho\left(f^r(x), f^r(y)\right) & \leq \diam{I_{k-j}^{m^{\prime \prime} - 1}} + \dist\left(I_{k-j}^{m^{\prime \prime} - 1}, I^{m^{\prime \prime}}_{j}\right) + \diam{I^{m^{\prime \prime}}_{j}} \leq \\&
\leq \frac{1}{2^{m^{\prime \prime}-1}}+
\dist\left(I_{k-j}^{m^{\prime \prime} - 1}, I_{1}^{m^{\prime \prime} - 1}\right)+
\dist\left(I_{1}^{m^{\prime \prime} - 1}, I^{m^{\prime \prime}}_{j}\right)+
\frac{1}{2^{m^{\prime \prime}}} \leq
\\&
\leq \frac{1}{2^{m^{\prime \prime} - 1}}+\frac{k-j}{l_{m^{\prime \prime}- 1}+1}+\frac{j}{l_{m^{\prime \prime}-1}+1}+\frac{1}{2^{m^{\prime \prime}}} \leq
\\&
\leq \frac{1}{2^{m^{\prime \prime}-2}}+\frac{k}{l_{m^{\prime \prime}-1}+1} \leq \frac{1}{2^{m^{\prime}-1}}+\frac{k}{l_{m^{\prime} }+1}<\eps.
\end{align*}
\end{enumerate}
Finally we get that for all $r\geq s+M$
$$\rho\left(f^r(x), f^r(y)\right) < \eps,$$
but $\eps$ can be arbitrarily small, thus
$$\lim_{r \to \infty} \rho \left(f^r(x), f^r(y) \right)=
0$$ and we get that $(x,y) \in \mathcal{D}\times \mathcal{D}$ is asymptotic.
\end{proof}

\begin{lem} \label{lem:wn}
There exists the sequence $(w_n)_{n=0}^{\infty}$ such that 
$$\lim_{n \to \infty} w_n=0$$ and  for each even $n$ and for $j \in \{1, 2,\dots, m_{n+1}\}$ 
$$\rho\left(f^{l_n+j}\left(\frac{1}{2},1\right),(1,0)\right) \leq  w_n $$
and for odd $n$ we have
$$\rho\left(f^{l_n+j}\left(\frac{1}{2},1\right),(0,0)\right) \leq  w_n .$$
\end{lem}
\begin{proof}
From (\ref{eq:defz}) we get that for every $k\in \N_0$ we have
$$\diam\left(z_{L_k}^{(k)}, 1\right) \leq \frac{1}{l_k+1} $$
and 
$$\diam\left(0,z_{1}^{(k)}\right) \leq \frac{1}{l_k+1}.$$ 

Note that 
$$\pi\left(f^{l_n+j}\left(\frac{1}{2},1\right)\right) \in \left(z_{L_n}^{(n)}, 1\right)\times \{0\} \subset I \times \{0\}$$
for $j=1, \dots, m_{n+1}$ and $n$ even. Now we have
\begin{align*} 
\rho & \left(f^{l_n+j}  \left(\frac{1}{2},1\right),(1,0)\right)  \leq \\&
 \leq \rho\left(f^{l_n+j}\left(\frac{1}{2},1\right),\pi\left(f^{l_n+j}\left(\frac{1}{2},1\right)\right)\right) +  \rho\left(\pi\left(f^{l_n+j}\left(\frac{1}{2},1\right)\right),(1,0)\right) \leq\\&
 \leq  \frac{1}{2^n}+\diam\left(z_{L_n}^{(n)},1\right) \leq \frac{1}{2^n} + \frac{1}{l_n+1}.
\end{align*}
When $n$ is odd we have that 
$$\pi\left(0,f^{l_n+j}\left(\frac{1}{2}\right)\right) \in \left(0,z_{1}^{(n)}\right)\times \{0\} \subset I\times \{0\}$$
 for all $j \in \{1, 2,\dots, m_{n+1}\}$ and
$$\rho\left(f^{l_n+j}\left(\frac{1}{2},1\right),(0,0)\right) \leq  \frac{1}{2^n} + \frac{1}{l_n+1}. $$
It is enough to take $$w_n=\frac{1}{l_n+1}+ \frac{1}{2^n}$$ for all $n\in \N_0$.  
\end{proof}

\begin{lem} \label{lem:DC1}
The map $f$ has DC1 pair.
\end{lem}
\begin{proof}

We will show that points $(\frac{1}{2},1)$ and $(1,0)$ form a  DC1 pair i.e. 
$$\liminf_{n \to \infty}\frac{1}{n}\left|\left\{0\leq m\leq n-1:\rho\left(f^m\left(\frac{1}{2},1\right),(1,0)\right)< s\right\}\right|=0$$ for some $s>0$ and
$$\limsup_{n \to \infty}\frac{1}{n}\left|\left\{0\leq m\leq n-1:\rho\left(f^m\left(\frac{1}{2},1\right),(1,0)\right) <t\right\}\right|=1$$ for all $t>0.$

From Lemma \ref{lem:wn}  we have for each even $n$ and for $j \in \{1, 2,\dots, m_{n+1}\}$ 
$$\rho\left(f^{l_n+j}\left(\frac{1}{2},1\right),(1,0)\right) \leq w_n $$
and for all sufficiently large odd $n$ we have
$$\rho\left(f^{l_n+j}\left(\frac{1}{2},1\right),(1,0)\right) > \frac{1}{2}.$$
For every $t$ we fix sufficiently large $N=N(t) \in \N$ then for all odd $n \geq N$ inequality holds
\begin{multline*}
\frac{\left| \left\{0\leq j\leq l_n+m_{n+1}-1:\rho\left(f^j\left(\frac{1}{2},1\right),(1,0)\right) <
w_N<t\right\}\right|}{l_n+m_{n+1}} \geq \\
\geq \frac{m_{n+1}}{l_n+m_{n+1}} \geq 1-\frac{l_n}{m_{n+1}} \geq 1- \frac{l_n}{2^n l_n}=1-\frac{1}{2^n}\to 1.
\end{multline*}

To show the second property of (DC1) let $s=\frac{1}{2}$ and then for all even $n \geq N$ we get

\begin{multline*}
\frac{\left| \left\{0\leq j\leq l_n+m_{n+1}-1:\rho\left(f^j\left(\frac{1}{2},1\right),(1,0)\right) <\frac{1}{2}\right\}\right|}{l_n+m_{n+1}} \leq \\\leq 1-\frac{m_{n+1}}{l_n+m_{n+1}} 
\leq \frac{l_n}{l_n+m_{n+1}} \leq \frac{l_n}{m_{n+1}} \leq \frac{l_n}{2^n l_n}=\frac{1}{2^n}\to 0.
\end{multline*}
It completes the proof.

\end{proof}

\begin{thm}\label{thm:dc1}
There exists a continuous self-map $f$ of dendrite such that:
\begin{enumerate}[(1)]
\item\label{coro:omega4:c1} $f$ has DC1 pair,
\item\label{coro:omega4:c2} $f$ does not have an infinite LY-scrambled set.
\end{enumerate}
\end{thm}
\begin{proof}
From Lemma \ref{lem:DC1} we have that $f$ has DC1 pair and 
from Lemma \ref{lem:dis} we get that there is no $(x, y, z)$ scrambled set. 
Indeed, let assume that $f^i(x) \notin I$ for  all $i\in \N.$ Now, if $(x,y)$ and $(x, z)$ are a $LY$-pairs then exist $n \in \N$ and $m \in \N$ such that $f^n(y)\in I$ and  $f^m(z) \in I$. 
We get that $(y, z)$ is not $LY$-pair because if $f^n(y) \neq f^m(z)$ then the distance between them is always positive and by the other hand if $f^n(y) =f^m(z)$ then the distance is always zero.
\end{proof}

\end{document}